\documentclass[a4paper,12pt]{article}
\usepackage{amsmath}
\usepackage{amsthm}
\theoremstyle{definition}
\newtheorem{theorem}{Theorem}
\newtheorem{proposition}[theorem]{Proposition}
\newtheorem{remark}[theorem]{Remark}
\usepackage[vvarbb]{newtx}
\usepackage[a4paper,truedimen,hmargin=1.5cm,vmargin=1.8cm]{geometry}
\usepackage[colorlinks,citecolor=blue]{hyperref}
\usepackage{cite}
\usepackage{bm}
\usepackage{algorithm,algpseudocode}
\usepackage{siunitx}
\usepackage{mathtools}
\mathtoolsset{showonlyrefs}
\usepackage{graphicx}

\newcommand{\bbC}{\mathbb{C}}
\newcommand{\bbR}{\mathbb{R}}
\newcommand{\bmA}{\bm{A}}
\newcommand{\bmb}{\bm{b}}

\newcommand{\bmc}{\bm{c}}
\newcommand{\bmC}{\bm{C}}
\newcommand{\bmD}{\bm{D}}

\newcommand{\bmH}{\bm{H}}
\newcommand{\bmI}{\bm{I}}

\newcommand{\bmK}{\bm{K}}
\newcommand{\bmM}{\bm{M}}

\newcommand{\bmu}{\bm{u}}
\newcommand{\bmv}{\bm{v}}
\newcommand{\bmV}{\bm{V}}
\newcommand{\bmx}{\bm{x}}

\newcommand{\calR}{\mathcal{R}}
\newcommand{\calW}{\mathcal{W}}
\newcommand{\rme}{\mathrm{e}}
\newcommand{\rmi}{\mathrm{i}}
\newcommand{\hatbmA}{\hat{\bmA}}

\author{
    Fuminori Tatsuoka\footnote{Independent Researcher (fuminori.tatsuoka@gmail.com).
    This research was conducted in a personal capacity and is independent of the author's current employer.}
    \and Yuto Miyatake\footnote{D3 Center, The University of Osaka (yuto.miyatake.cmc@osaka-u.ac.jp).}
    \and Tomohiro Sogabe\footnote{Department of Applied Physics, Graduate School of Engineering, Nagoya University (sogabe@na.nuap.nagoya-u.ac.jp)}
}
\title{An Error Control Framework for Computing the Exponential of Matrices Arising from the Finite Element Discretization}

\begin{document}
\maketitle

\begin{abstract}
  Several methods for computing the action of the matrix exponential $\rme^{\bmA} \bmb$ are expressed by substituting $\bmA$ into a rational approximation of the scalar exponential function.
  When $\bmA$ is a general sparse matrix, the approximation error of such methods can be bounded using the numerical range of $\bmA$, which enables the computation of $\rme^{\bmA} \bmb$ with a prescribed accuracy.
  This study extends this error control framework to matrices of the form $\bmA = \tau \bmM^{-1} \bmK$, where $\bmM$ is a well-conditioned symmetric positive definite matrix.
  A direct application of the framework faces two challenges: determining the numerical range of $\bmA$ is not straightforward, and the resulting error bounds can significantly overestimate the approximation error.
  In this paper, we show that these issues can be resolved by using the numerical range of the similarity transformed version of $\bmA$.
  Numerical experiments confirm that the computations can be performed within the prescribed error tolerance.
\end{abstract}

\section{Introduction}
For a square matrix $\bmA \in \bbC^{n\times n}$, its exponential is defined as \cite[Eq. (10.1)]{higham2008functions}
\begin{align}
  \rme^{\bmA} \coloneq \bmI + \bmA + \frac{1}{2!}\bmA^2 + \frac{1}{3!}\bmA^3 + \cdots.
\end{align}
The matrix exponential allows us to represent the solutions of linear initial value problems:
for the initial value problem
\begin{align}
  \dot{\bmu}(t) = \bmA \bmu(t),
  \qquad \bmu(0) = \bmu_0,
\end{align}
it holds that
\begin{align}
  \bmu(t) = \rme^{t\bmA} \bmu_0.
\end{align}
This relation facilitates the construction of numerical solvers called exponential integrators for problems that are challenging to solve with conventional methods such as Runge--Kutta methods; see, e.g., \cite{hochbruck2010exponential} for more details.
Hence, computational methods for $\rme^{\bmA}\bmb ~ (\bmb\in\bbC^n)$ are of significant importance in numerical analysis.

Roughly speaking, computational methods for $\rme^{\bmA}\,\bmb$ can be classified into two categories.
The first category includes methods that substitute $\bmA$ into a prescribed scalar rational approximation $r(z) \approx \rme^{z}$.
Although methods in this category are well established, they remain a subject of active research for higher efficiency and reliability.
Examples of such methods include the truncated Taylor series \cite{almohy2011computing,ibanez2022two}
\begin{align}
  \rme^{\bmA}\bmb
  \approx \left(
    \bmI + \bmA + \frac{1}{2!}\bmA^2 + \frac{1}{3!}\bmA^3 + \cdots + \frac{1}{m!}\bmA^m
  \right) \bmb,
\end{align}
as well as diagonal and sub-diagonal Pad\'{e} approximation \cite{guttel2016scaled,blanes2026parallel}, polynomial interpolation \cite{caliari2004interpolating,caliari2016leja}, and numerical quadrature \cite{weideman2006optimizing,weideman2007parabolic,tatsuoka2024computing}.

The other category includes methods based on projection on an appropriate low dimensional subspace.
Let $\bmV \in \bbC^{n \times m} ~ (m \ll n)$ be an orthogonal basis of a certain subspace.
Then, $\rme^{\bmA} \bmb$ is approximated by $\bmV \rme^{\bmH} \bmV^\ast \bmb$, where $\bmH = \bmV^\ast \bmA \bmV \in \bbC^{m\times m}$.
Common choices for $\bmV$ that are frequently used are polynomial, extended or rational Krylov subspaces \cite{saad1992analysis,gockler2013convergence,guttel2013rational}.

While methods in the latter category are often employed due to their efficiency in the context of exponential integrators, we focus on methods in the former category due to several advantages.
First, some of these methods are applicable even in limited-memory environments.
In contrast, projection-based methods generally require storing the basis matrix $\bmV \in \bbC^{n \times m}$, which is a limitation in terms of memory requirements.
Although restarted variants with a fixed amount of storage space have been proposed \cite{eiermann2006restarted,botchev2020art}, there is typically a trade-off where reducing memory consumption increases the computational cost.
Second, when $r(z)$ is given in partial fraction form (e.g., the subdiagonal Pad\'{e} approximation and quadrature-based algorithms), the resulting algorithm is well-suited for parallel computing.
Third, because the scalar approximation $r(z)$ is constructed independent of $\bmb$, once a rational function satisfying $\|r(\bmA) - \rme^{\bmA}\|_2 \le \epsilon$ is determined, the same rational function can be reused for arbitrary vectors $\bmb$.
This property is particularly beneficial in exponential integrators where $\rme^{\tau \bmA} \bmb$ must be evaluated repeatedly at each time step with varying vectors $\bmb$.

To guarantee the approximation accuracy, it is important to check if $\|r(\bmA) - \rme^{\bmA}\|_2 \le \epsilon$ is satisfied for a given tolerance $\epsilon$.
A uniform approach to bound $\|r(\bmA) - \rme^{\bmA}\|_2$ for various $r$ is to utilize the numerical range (also called the field of values) of $\bmA$: $\calW(\bmA) \coloneq \{\bmv^\ast \bmA \bmv \colon \bmv \in \bbC^n, \|\bmv\|_2 = 1\}$.
It is known that $\calW(\bmA)$ is a $(1+\sqrt{2})$-spectral set for $\bmA$ (see \cite{crouzeix2017numerical}), and therefore, for any rational function $r$ that is bounded on $\calW(\bmA)$, it holds that
\begin{align} \label{eq:crouzeix_inequality}
  \|r(\bmA) - \rme^{\bmA}\|_2 \le (1 + \sqrt{2}) \sup_{z \in \calW(\bmA)} |r(z) - \rme^z|.
\end{align}
While the numerical range $\calW(\bmA)$ cannot be obtained analytically and estimating it precisely is computationally expensive in general, a rectangular region $\calR$ that contains $\calW(\bmA)$ can be obtained with cheaper costs.
Specifically, the left and right endpoints of $\calR$ are given by the extreme eigenvalues of $(\bmA + \bmA^\ast)/2$, and the top and bottom endpoints correspond to the extreme eigenvalues of $(\bmA - \bmA^\ast)/2 \rmi$.
This fact is noted in e.g., \cite{caliari2016leja} and is a special case of the algorithm estimating $\calW(\bmA)$ in \cite{johnson1978numerical}.

Based on the above discussion, an error control framework for $\rme^{\bmA} \bmb$ can be formulated as in Algorithm~\ref{alg:conventional}.
In this framework, we first determine $\calR$ enclosing $\calW(\bmA)$ by computing the extreme eigenvalues of the Hermitian and skew-Hermitian parts of $\bmA$.
We then construct $r(z) \approx \rme^z$ on $\calR$ such that the condition $\|r(\bmA) - \rme^{\bmA}\|_2 \le \epsilon$ is satisfied.
We note that Algorithm \ref{alg:conventional} is not a single specific algorithm; rather, it acts as a framework that generates various error controlled algorithms depending on the choice of $r$.
For example, a similar strategy has been utilized to provide an error bound of polynomial interpolation \cite{caliari2004interpolating} \footnote{
  Although the $\phi$-function $(\rme^z - 1)/z$ is focused  in \cite{caliari2004interpolating} rather than the exponential function, the method described therein remains applicable to the exponential, see \cite{caliari2016leja}.
}.
Under this framework, even when an error analysis for the matrix error $r(\bmA) - \rme^{\bmA}$ of a given approximation $r$ (such as rational and polynomial interpolations) is difficult, one can rigorously guarantee the accuracy by ensuring scalar level accuracy over the bounding box $\calR$.
\begin{algorithm}[htbp]
  \caption{An error control framework based on numerical range for $\rme^{\bmA} \bmb$.}
  \label{alg:conventional}
  \begin{algorithmic}[1]
    \Statex \textbf{Input} $\bmA, \bmb, \epsilon$
    \State Compute the maximum (minimum) eigenvalue $\mu_{\max}$ ($\mu_{\min}$) of $(\bmA + \bmA^\ast)/2$
    \State Compute the maximum (minimum) eigenvalue $\nu_{\max}$ ($\nu_{\min}$) of $(\bmA - \bmA^\ast)/2\rmi$
    \State $\calR = \left\{z \in \bbC \colon \mathrm{Re}(z) \in [\mu_{\min},\mu_{\max}], \mathrm{Im}(z) \in [\nu_{\min},\nu_{\max}] \right\} \quad ( \supseteq \calW(\bmA))$
    \State Construct a rational approximation $r(z) \approx \rme^z$ on $\calR$ such that
    \begin{align}
      \max_{z \in \calR} |r(z) - \rme^z| \le \frac{\epsilon}{1 + \sqrt{2}}.
    \end{align}
    \Statex \textbf{Output} $r(\bmA)\bmb \quad (\|r(\bmA)\bmb - \rme^{\bmA}\bmb\|_2 \le \epsilon \|\bmb\|_2)$
  \end{algorithmic}
\end{algorithm}

In this study, we focus on the situation $\bmA = \tau \bmM^{-1} \bmK$, where $\tau > 0$, $\bmM\in\mathbb{R}^{n\times n}$ is a well-conditioned symmetric positive definite (SPD) matrix, and $\bmK\in\mathbb{R}^{n\times n}$.
For example, such matrices arise from the finite element discretization of advection-diffusion equations.
There are two difficulties to apply Algorithm \ref{alg:conventional} to this situation directly.
First, to the best of our knowledge, there are no estimates and computational methods for obtaining eigenvalues of the (skew-)symmetric part of $\bmA = \tau \bmM^{-1} \bmK$, and therefore computing the rectangular region $\calR$ is not straightforward.
Second, $\calW(\bmA)$ can be so large that an accurate rational approximation cannot be constructed on it.
For example, the left figure in Figure \ref{fig:example_nr} illustrates $\calW(\bmA)$ for a test matrix from an advection-diffusion equation described in Section 3.
The figure shows that $\calW(\bmA)$ extends into the right half plane.
For this case, $r(z)$ must deal with large values (e.g., $\rme^{17.29}$) with the error tolerance $\epsilon$, which is practically impossible.

\begin{figure}[htbp]
  \centering
  \includegraphics[width=12cm]{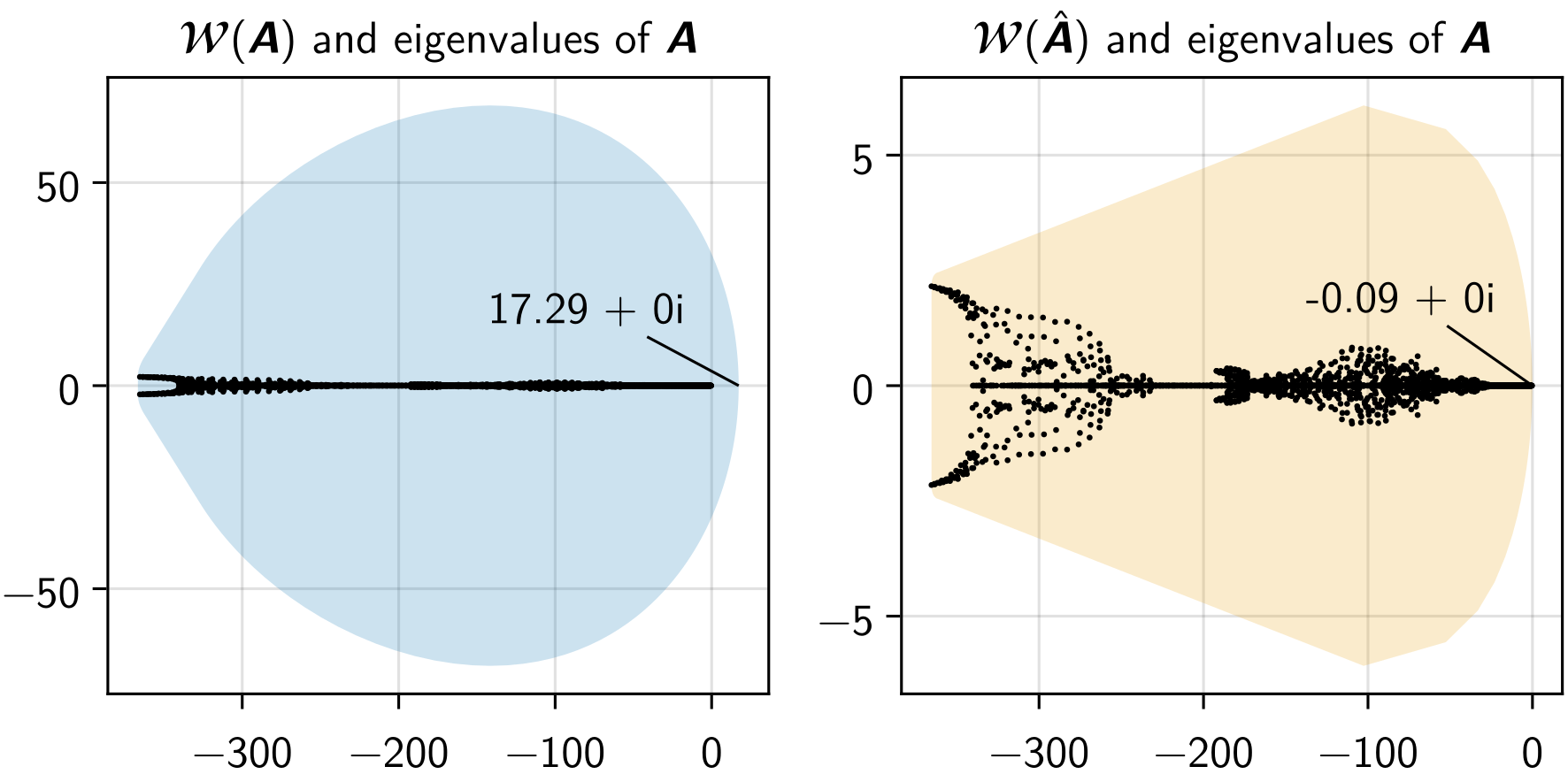}
  \caption{
    Comparison of $\calW(\bmA)$ and $\calW(\hatbmA)$.
    The matrix $\bmA$ is a test matrix described in Section 3.1, obtained by discretizing the two dimensional advection-diffusion equation on a square domain with a P2 finite element method.
    The set $\calW(\hatbmA)$ lies in the left half plane.
  }
  \label{fig:example_nr}
\end{figure}

In this paper, we extend the error control framework to $\bmA = \tau \bmM^{-1} \bmK$.
That is, rather than proposing a single algorithm for $\rme^{\tau \bmM^{-1} \bmK}$, we aim to make rational approximation-based algorithms practically applicable to $\bmA = \tau \bmM^{-1} \bmK$.
The main idea to overcome the difficulties is to use $\calW(\hatbmA)$ instead of $\calW(\bmA)$, where $\hatbmA \coloneq \bmM^{1/2} \bmA \bmM^{-1/2}$.
There are three convenient properties for $\calW(\hatbmA)$.
First, $\calW(\hatbmA)$ also provides us with an upper bound on $\|r(\bmA) - \rme^{\bmA}\|_2$, that is, there is a constant $C>0$ depending only on $\bmA$ such that $\|r(\bmA) - \rme^{\bmA}\|_2 \le C \sup_{z \in \calW(\hatbmA)} |r(z) - \rme^z|$.
Second, a rectangular region enclosing $\calW(\hatbmA)$ can be obtained numerically by solving typical generalized eigenvalue problems.
Third, if $\calW(\bmK)$ lies in the left half plane, then $\calW(\hatbmA)$ still lies in the left half plane.
The third property can be seen in the right figure in Fig.~\ref{fig:example_nr}.
Based on them, we establish an error control framework for computing $\rme^{\tau \bmM^{-1} \bmK}\bmb$.

The organization of this paper is as follows.
Section 2 presents the error control framework for $\bmA = \tau\bmM^{-1}\bmK$.
Section 3 provides numerical results, and Section 4 concludes this paper.

\section{Error control framework\texorpdfstring{ for $\bmA = \tau \bmM^{-1} \bmK$}{}}
\subsection{Framework}
In this section, we show some properties of $\calW(\hatbmA)$ and propose the error control framework.
Firstly, it is shown that $\calW(\hatbmA)$ gives us an upper bound on the error $r(\bmA) - \rme^{\bmA}$ as follows:
\begin{theorem}
  For $\bmA = \tau \bmM^{-1}\bmK$, where $\tau > 0$, $\bmM \in \bbR^{n\times n}$ is an SPD matrix, and $\bmK \in \bbR^{n\times n}$, define $\hatbmA = \bmM^{1/2}\bmA \bmM^{-1/2}$.
  Then, for any rational function $r(z)$ that is bounded on $\calW(\hatbmA)$, it follows that
  \begin{align} \label{eq:new_bound}
    \|r(\bmA) - \rme^{\bmA}\|_2 \le (1+\sqrt{2}) \kappa(\bmM)^{1/2} \sup_{z \in \calW(\hatbmA)} |r(z) - \rme^z|,
  \end{align}
  where $\kappa(\bmM) = \|\bmM\|_2\|\bmM^{-1}\|_2$.
\end{theorem}

\begin{proof}
  For notational simplicity we write $g(\bmA)=r(\bmA)-\rme^{\bmA}$.
  Using the property in \cite[Thm.~1.13 (c)]{higham2008functions}, we can rewrite $g(\bmA)$ as
  \begin{align}
    g(\bmA) = g(\bmM^{-1/2}\bmM^{1/2}\bmA\bmM^{-1/2}\bmM^{1/2}) = \bmM^{-1/2} g(\bmM^{1/2}\bmA\bmM^{-1/2}) \bmM^{1/2} = \bmM^{-1/2} g(\hatbmA) \bmM^{1/2}.
  \end{align}
  Thus, we have
  \begin{align}
    \|g(\bmA)\|_2 = \|\bmM^{-1/2} g(\hatbmA) \bmM^{1/2}\|_2 \le \kappa(\bmM^{1/2}) \|g(\hatbmA)\|_2
  \end{align}
  because $\kappa(\bmM^{1/2}) = \kappa(\bmM)^{1/2}$.
  By using \eqref{eq:crouzeix_inequality} with the matrix $\hatbmA$, we have \eqref{eq:new_bound}.
\end{proof}

\begin{remark}
  Because the eigenvalues of $\hatbmA$ coincide with those of $\bmA$, $\calW(\hatbmA)$ is a $(1+\sqrt{2})\kappa(\bmM)^{1/2}$-spectral set of $\bmA$.
  See, e.g., \cite{crouzeix2017numerical} for the definition of spectral sets.
\end{remark}

Next, we show that a rectangular region enclosing $\calW(\hatbmA)$ can be obtained by solving generalized eigenvalue problems as follows:
\begin{proposition}
  Let us define
  \begin{align}
    \bmD \coloneq \frac{1}{2}(\bmK + \bmK^\top), \qquad \bmC \coloneq \frac{1}{2 \rmi}(\bmK - \bmK^\top).
  \end{align}
  Then, the left (right) endpoint of $\calW(\hatbmA)$ is the minimum (maximum) eigenvalue of the generalized eigenvalue problem $(\tau \bmD) \bmx = \lambda \bmM \bmx$.
  Similarly the top (bottom) endpoint of $\calW(\hatbmA)$ is the maximum (minimum) eigenvalue of the generalized eigenvalue problem $(\tau \bmC)\bmx = \lambda \bmM \bmx$.
\end{proposition}

\begin{proof}
  For $\hatbmA = \bmM^{1/2} \bmA \bmM^{-1/2} = \tau \bmM^{-1/2}\bmK\bmM^{-1/2}$, its symmetric and skew-symmetric parts are $\tau \bmM^{-1/2}\bmD\bmM^{-1/2}$ and $\tau \bmM^{-1/2}\bmC\bmM^{-1/2}$, respectively because $\bmM$ is SPD.
  Therefore, the eigenvalues of the symmetric part coincide with those of $\bmM^{-1}(\tau\bmD)$, and the eigenvalues of the skew-symmetric part coincide with those of $\bmM^{-1}(\tau \bmC)$.
\end{proof}

It is known that the numerical range of the advection-diffusion operator lies in the left half plane.
In addition, if the finite element method is employed, the numerical range of the resulting matrix $\bmK$ also lies in the left half plane \cite{gockler2017acceleration}.
The set $\calW(\hatbmA)$ inherits this convenient property as follows:
\begin{proposition}
  Suppose that $\calW(\bmK)$ is in the left half plane.
  Then, $\calW(\hatbmA)$ is also in the left half plane.
\end{proposition}

\begin{proof}
  For $\bmv \in \bbC^{n}$, we have
  \begin{align}
    \mathrm{Re}(\bmv^\ast \hatbmA \bmv) = \mathrm{Re}(\tau \bmv^\ast \bmM^{-1/2} \bmK \bmM^{-1/2} \bmv)
    = \tau \mathrm{Re}\left((\bmM^{-1/2}\bmv)^\ast \bmK (\bmM^{-1/2} \bmv)\right) \le 0.
  \end{align}
\end{proof}

At the end of this subsection, we summarize the error control framework in Algorithm \ref{alg:propose}.
First, we determine a rectangle $\calR$ enclosing $\calW(\hatbmA)$ by computing the extreme eigenvalues of $(\tau\bmD, \bmM)$ and $(\tau\bmC, \bmM)$, respectively.
Then, following a procedure similar to Algorithm \ref{alg:conventional}, we construct a rational function $r(z) \approx \rme^z$ on $\calR$ and compute $r(\bmA)\bmb$.

Note that the eigenvalues and the condition number required in Algorithm \ref{alg:propose} only need to be estimated roughly.
For example, it suffices to obtain a lower bound for $\kappa(\bmM)$.
If $\kappa(\bmM)$ is computed with a relative error bounded by $\delta~ (< 1)$, yielding an approximate value $\tilde{\kappa}$, one can use $\tilde{\kappa}/(1-\delta)$ in place of $\kappa(\bmM)$ because the inequality $\kappa(\bmM) \le \tilde{\kappa}/(1-\delta)$ holds.

Therefore, the evaluation of $r(\bmA)\bmb$ dominates the computational cost of Algorithm~\ref{alg:propose}.
Although estimating $\kappa(\bmM)$ and the generalized eigenvalues are also required, their computational cost is relatively small because they only require rough computation of eigenvalues (without eigenvectors) for Hermitian matrices.

\begin{algorithm}[htbp]
  \caption{The proposed error control framework for $\rme^{\bmA} \bmb$ with $\bmA = \tau \bmM^{-1}\bmK$}
  \label{alg:propose}
  \begin{algorithmic}[1]
    \Statex \textbf{Input} $\tau, \bmM, \bmK, \bmb, \epsilon$
    \State $\bmD = (\bmK + \bmK^\top)/2, \quad \bmC = (\bmK - \bmK^\top) / 2\rmi$
    \State Compute the maximum (minimum) eigenvalue $\mu_{\max}$ ($\mu_{\min}$) of the generalized eigenvalue problem $(\tau \bmD) \bmx = \mu \bmM \bmx$.
    \State Compute the maximum (minimum) eigenvalue $\nu_{\max}$ ($\nu_{\min}$) of the generalized eigenvalue problem $(\tau \bmC) \bmx = \mu \bmM \bmx$.
    \State $\calR = \{z \in \bbC \colon \mathrm{Re}(z) \in [\mu_{\min}, \mu_{\max}], ~ \mathrm{Im}(z) \in [\nu_{\min}, \nu_{\max}]\} \quad (\supseteq \calW(\hatbmA))$
    \State Compute $\kappa(\bmM)$
    \State Construct a rational approximation $r(z) \approx \rme^z$ on $\calR$ such that
    \begin{align}
      \max_{z \in \calR} |r(z) - \rme^z| \le \frac{\epsilon}{(1 + \sqrt{2}) \kappa(\bmM)^{1/2}}.
    \end{align}
    \Statex \textbf{Output} $r(A)\bmb \quad (\|r(\bmA)\bmb - \rme^{\bmA}\bmb\|_2 \le \epsilon \|\bmb\|_2)$
  \end{algorithmic}
\end{algorithm}

\subsection{Implementation examples of Algorithm 2}
When implementing Algorithm 2, it is necessary to prepare a rational function that well-approximates the exponential function on $\calW(\hatbmA)$.
Although various rational approximation methods may be applied, in the numerical example in the next section, we consider subdiagonal Pad\'{e} approximation and rational interpolation for the following two reasons.
First, because $\|\bmA\|_2$ is not necessarily small, the standard scaling and squaring technique can be inefficient.
In diagonal Pad\'{e} or Taylor approximations, the scaling parameter $s$ is typically chosen such that $\|\bmA / s\|$ is about 1 for some matrix norm.
In this case, since $\rme^{\bmA} \bmb = (\rme^{\bmA / s})^s \bmb$, the action of the rational approximant is computed $s$ times.
In exponential integrators, it is desirable to choose the time step size $\tau$ as large as possible, so $\|\bmA\|$ can also become large in practical applications.
Indeed, some of the test matrices satisfy $\|\bmA\|_2 > 10^3$, which can lead to high computational costs.
Second, because $\bmA$ involves $\bmM^{-1}$, the matrix--vector product $\bmA \bmv ~ (\bmv \in \mathbb{C}^n)$ itself requires solving a linear system.
Consequently, the computational cost of $\bmA \bmv$ is closer to that of solving the (shifted) linear system $(\bmA + \sigma \bmI)^{-1}\bmv = (\tau \bmK + \sigma \bmM)^{-1}\bmM \bmv$ ($\sigma \in \mathbb{C}$) than in the case of standard sparse matrices.
Both of the rational approximations chosen here admit partial fraction expansions; thus, depending on the computing environment, the corresponding linear systems can be solved in parallel, potentially reducing computation time.
Nevertheless, whether these two rational approximation methods are optimal depends on the choice of linear solver and the computing environment, and a comprehensive investigation of this aspect is left for future work.

\begin{itemize}
  \item \textbf{Subdiagonal Pad\'{e} approximation} (\texttt{sub-pade})\\
    The Pad\'{e} approximation of $\rme^{z}$ at $z=0$ has been widely used for computing matrix exponentials and their actions, as its coefficients are available in closed form.
    Although diagonal \cite{higham2005scaling,almohy2010new} or polynomial \cite{almohy2011computing} types are common choices, G\"{u}ttel and Nakatsukasa \cite{guttel2016scaled} pointed out several advantages of subdiagonal types.
    For example, the subdiagonal Pad\'e approximation may reduce the computational cost for matrices having large eigenvalues in the negative direction (owing to the smaller scaling parameter required), and it can be implemented in partial fraction form.
    Given the lack of rigorous error estimates for non-symmetric matrices in \cite{guttel2016scaled}, we employ the following heuristic parameter selection:
    \begin{itemize}
      \item The Pad\'{e} order is fixed to $(4,5)$.
      \item The scaling parameter $s_\ast$ is chosen to ensure that the approximation error on the rectangle $\calR$ is within the target tolerance:
      \begin{align}
        s_{\ast} = \min \left\{
        s = 1,2,3,\dots \colon
        \sup_{z\in\calR} |\rme^{z}-(r_{4,5}(z/s))^{s}| \le \frac{\epsilon}{(1+\sqrt{2}) \kappa(\bmM)^{1/2}}
        \right\},
      \end{align}
      where $r_{4,5}$ is the $(4,5)$-type Pad\'{e} approximant.
      \item The approximation $r_{4,5}(\bmA/s) \cdots r_{4,5}(\bmA/s) \bmb$ is computed in the partial fraction form, see \cite[Sect.~3.2]{guttel2016scaled}.
    \end{itemize}
  \item \textbf{Rational interpolation} (\texttt{rat-interp})\\
  Recent advances in algorithms for rational approximation of scalar functions enable the efficient construction of approximations over specified complex domains.
  An outline of the method \texttt{rat-interp} is given below, with its technical details described in Appendix~\ref{sect:rational_interpolation}.
  \begin{itemize}
    \item First, using the AAA algorithm, we numerically obtained a barycentric rational approximation
    \begin{align}\label{eq:barycentric}
      r_{\mathrm{AAA}}(z) = \frac{p(z)}{q(z)},
      \qquad p(z) = \sum_{k=1}^m \frac{w_k \exp(\zeta_k)}{z - \zeta_k},
      \qquad q(z) = \sum_{k=1}^m \frac{w_k}{z - \zeta_k}.
    \end{align}
    to $\rme^z$ on $\calR$ with an error of $\epsilon / (1+\sqrt{2}) \kappa(\bmM)^{1/2}$ or less.
    \item Next, we converted $r_{\mathrm{AAA}}$ into the partial fraction form:
    \begin{align}
      r_{\mathrm{pf}}(z) = \gamma + \sum_{k=1}^m \frac{\alpha_k}{z - \beta_k}.
    \end{align}
    \begin{itemize}
      \item We computed the poles of $r_{\mathrm{AAA}}$ and used them as $\beta_k$.
      \item Then, we deterimne $\gamma$ and $\alpha_k$ by the linear least squares method.
      Because the coefficient matrix in the least square problem can be ill-conditioned, we employ the double-double arithmetic for solving it.
    \end{itemize}
    \item Finally, we computed
    \begin{align}
      r_{\mathrm{pf}}(\bmA) \bmb = \gamma \bmb + \sum_{k=1}^m \alpha_k (\tau \bmK - \beta_k \bmM)^{-1} (\bmM \bmb).
    \end{align}
  \end{itemize}
\end{itemize}

\section{Numerical examples}
This section presents numerical results of the proposed framework described in Algorithm \ref{alg:propose}.
The experiments were conducted on a Mac Studio M4 Max equipped with 128 GB of RAM.
All algorithms were implemented in Julia 1.12.3, and the source code is available from the corresponding author upon request.
Unless otherwise specified, all computations were performed using IEEE double-precision arithmetic.

\subsection{Test problems}
We consider a two-dimensional advection-diffusion equation on both a square domain and a star-shaped domain, as illustrated in Figure \ref{fig:domain}.
The governing equation is given by
\begin{align}\label{eq:test_problem}
  \begin{aligned}
    & \frac{\partial u}{\partial t} = d \Delta u + \bmc^\top \nabla u && ((x,y) \in \Omega, t > 0),\\
    & u(t,x,y) = 0 && ((x,y) \in \partial\Omega),\\
    & u(0,x,y) = u_0(x,y) && ((x,y) \in \Omega, t=0),\\
  \end{aligned}
\end{align}
where $d \in \{10^{-1}, 10^{-3}\}$, $\bmc = [1, 1]^\top$.
The initial condition for the square domain is
\begin{align}
  u_0(x,y) = \exp\left(-\sinh(70(x-1/2)^4) -\sinh(70(y-1/2)^4)\right),
\end{align}
and that for the star-shaped domain is
\begin{align}
  u_0(x,y) = \exp\left(-\sinh(70(x/2)^4) -\sinh(70(y/2)^4)\right).
\end{align}
We employ the standard Galerkin discretization with the mesh generated by FreeFEM++ \cite{hecht2012new}.
This yields the following system of ordinary differential equations as a semi-discrete scheme:
\begin{align}
  \bmM \bmu'(t) = \bmK\bmu(t), \qquad \bmu(0) = \bmb.
\end{align}
An exponential integrator requires to compute $\rme^{\tau \bmM^{-1}\bmK} \bmb$ for a given time step size $\tau$ ($\bmb$ is replaced with another vector as time integration proceeds).
Since our aim is to check the performance of the proposed method, we employ the linear PDE as our toy problems.

\begin{figure}[htbp]
  \centering
  \includegraphics[width=12cm]{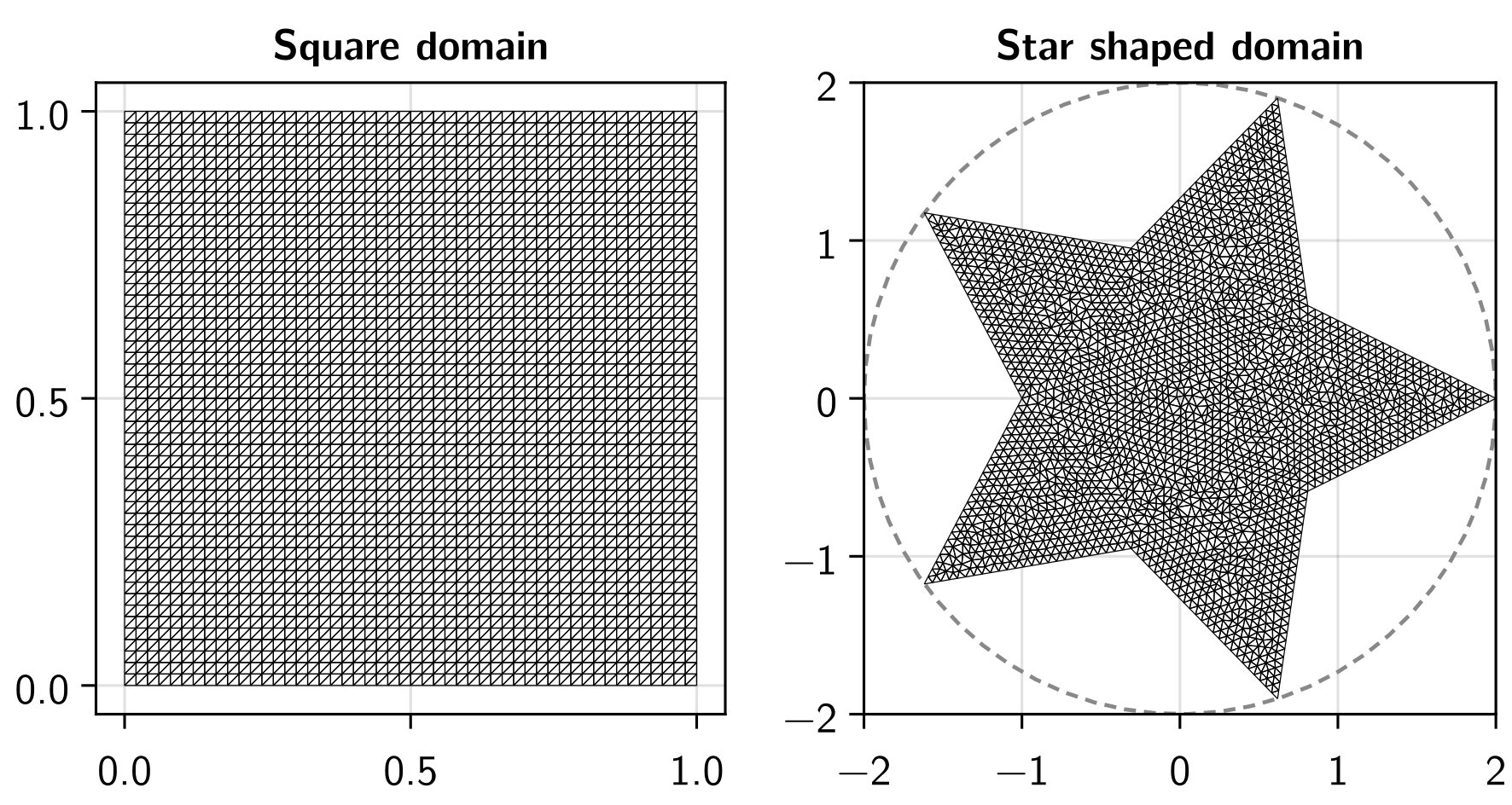}
  \caption{The square domain and the star-shaped domain we used.}
  \label{fig:domain}
\end{figure}

\subsection{Results for checking mathematical validity}
Using the medium size matrices, we see several advantages of using $\calW(\hatbmA)$ rather than $\calW(\bmA)$.
The advection-diffusion equation \eqref{eq:test_problem} is discretized such that the matrix size $n$ is approximately 2000.
The time step $\tau$ is set to the average mesh edge length $\bar{h}$.
Table~\ref{tab:test_matrices} summarizes the properties of the test matrices.
Their numerical ranges, $\calW(\bmA)$ and $\calW(\hatbmA)$, are visualized in Figure \ref{fig:numerical_range_test_matrices}.
The details of computational methods are explained in Appendix \ref{sect:nr}.
As shown in Proposition~2, $\calW(\hatbmA)$ lies in the left half plane while $\calW(\bmA)$ extends into the right half plane.
\begin{table}[htbp]
  \centering

  \caption{
    Properties of the test matrices $\bmA = \tau \bmM^{-1}\bmK$.
    Here, $n$ denotes the matrix size and $\bar{h}$ is the average edge length of the mesh.
  }

  \begin{tabular}{ccc|rrr}
    \hline
    Shape & Element & $d$ & $n$ & $\bar{h}$ & $\kappa(\bmM)$ \\
    \hline
    square & P1 & \num{e-1} & 2401 & \num{2.28e-02} & \num{3.99e+00} \\
    square & P1 & \num{e-3} & 2401 & \num{2.28e-02} & \num{3.99e+00} \\
    square & P2 & \num{e-1} & 2401 & \num{4.55e-02} & \num{5.81e+00} \\
    square & P2 & \num{e-3} & 2401 & \num{4.55e-02} & \num{5.81e+00} \\
    star   & P1 & \num{e-1} & 2192 & \num{5.51e-02} & \num{6.55e+00} \\
    star   & P1 & \num{e-3} & 2192 & \num{5.51e-02} & \num{6.55e+00} \\
    star   & P2 & \num{e-1} & 2431 & \num{1.05e-01} & \num{1.17e+01} \\
    star   & P2 & \num{e-3} & 2431 & \num{1.05e-01} & \num{1.17e+01} \\
    \hline
    \hline
    Shape & Element & $d$ &  $\|\bmA\|_2 ~ (\tau=\bar{h})$ & $\|\rme^{\bmA} \bmb\|_2 ~ (\tau=\bar{h})$ & $\|\rme^{\bmA} \bmb\|_2 ~ (\tau=10\bar{h})$ \\
    \hline
    square & P1 & \num{e-1} & \num{1.47e+02} & \num{9.74e-03} & \num{5.37e-01} \\
    square & P1 & \num{e-3} & \num{6.92e+01} & \num{5.18e-02} & \num{6.76e-01} \\
    square & P2 & \num{e-1} & \num{3.68e+02} & \num{1.23e-02} & \num{2.24e-01} \\
    square & P2 & \num{e-3} & \num{1.92e+02} & \num{6.31e-02} & \num{3.27e-01} \\
    star   & P1 & \num{e-1} & \num{2.55e+00} & \num{1.03e-02} & \num{9.86e-01} \\
    star   & P1 & \num{e-3} & \num{3.74e+00} & \num{5.42e-02} & \num{1.03e+00} \\
    star   & P2 & \num{e-1} & \num{7.04e+00} & \num{1.36e-02} & \num{6.34e-01} \\
    star   & P2 & \num{e-3} & \num{1.01e+01} & \num{6.62e-02} & \num{7.79e-01} \\
    \hline
  \end{tabular}
  \label{tab:test_matrices}
\end{table}

\begin{figure}[htbp]
  \centering
  \includegraphics[height=22cm]{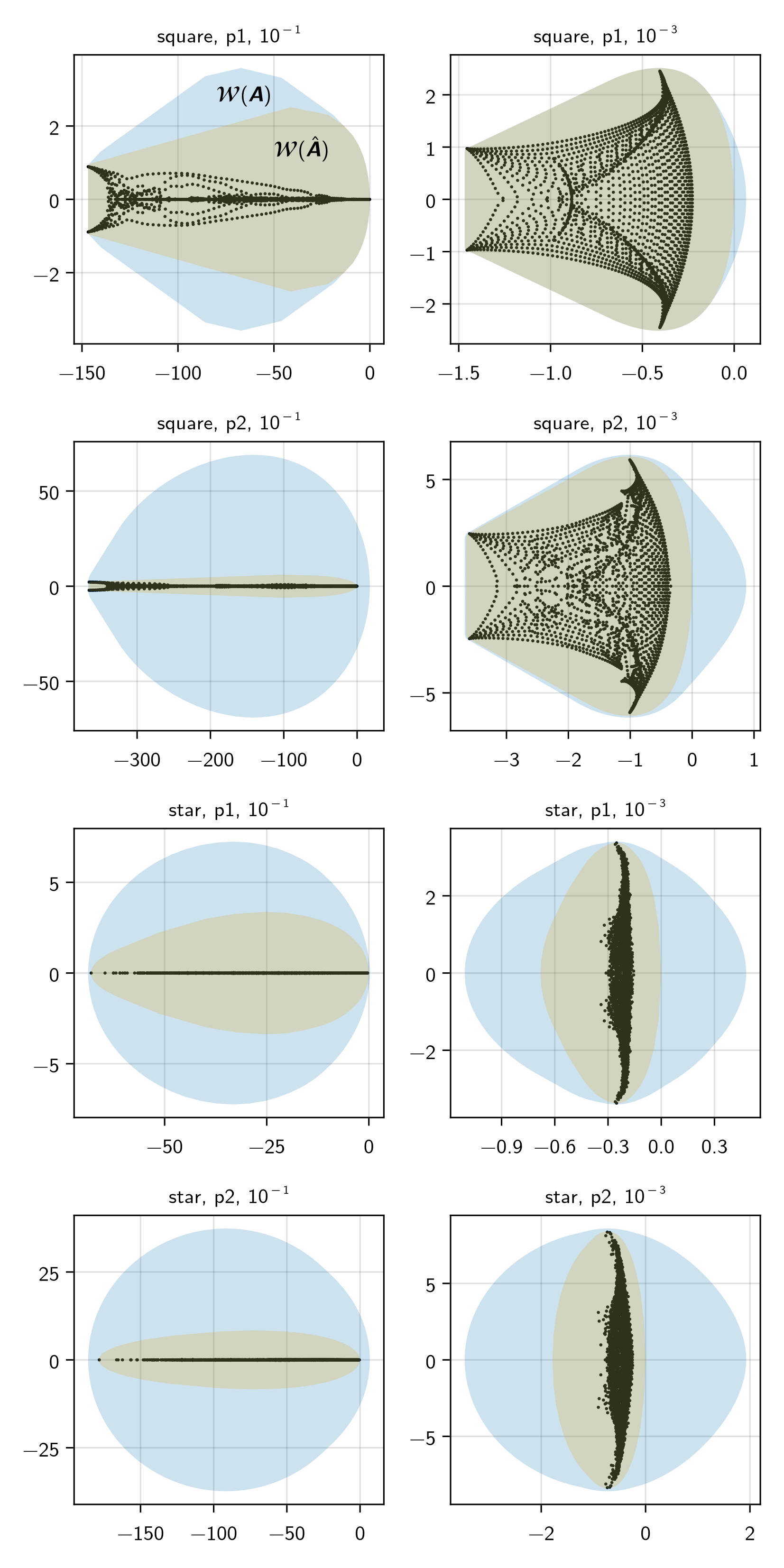}
  \caption{
    Eigenvalues and numerical ranges $\calW(\bmA)$, $\calW(\hatbmA)$ for the test matrices in Tab.~\ref{tab:test_matrices}.
    Black dots represent eigenvalues, while blue and brown regions denote $\calW(\bmA)$ and $\calW(\hatbmA)$, respectively.
    For these matrices, $\calW(\hatbmA)$ is contained within $\calW(\bmA)$.
  }
  \label{fig:numerical_range_test_matrices}
\end{figure}

In the computation, eigenvalues for the pencils $(\tau\bmD, \bmM)$ and $(\tau\bmC,\bmM)$ were computed using the \texttt{eigvals} function that computes all eigenvalues.
The condition number $\kappa(\bmM)$ was obtained via the \texttt{cond} function that computes all singular values.
Linear systems were solved using a direct solver (\texttt{K\textbackslash b} in Julia).
Reference solutions were generated by \texttt{exp(tau*(M\textbackslash K))*b}, where \texttt{exp} is Julia's built-in function, which is implemented based on the scaling and squaring algorithm in \cite{higham2005scaling}.
The target tolerance $\epsilon$ was varied from $10^{-8}$ to $10^{-2}$.

Figure \ref{fig:result} shows the approximation errors for each method.
Each data point corresponds to a matrix from Table \ref{tab:test_matrices}.
While some cases exhibit overestimation, the error remains below the prescribed tolerance in all instances, confirming that the proposed error control framework functions as intended.
\begin{figure}[htbp]
  \centering
  \includegraphics[width=12cm]{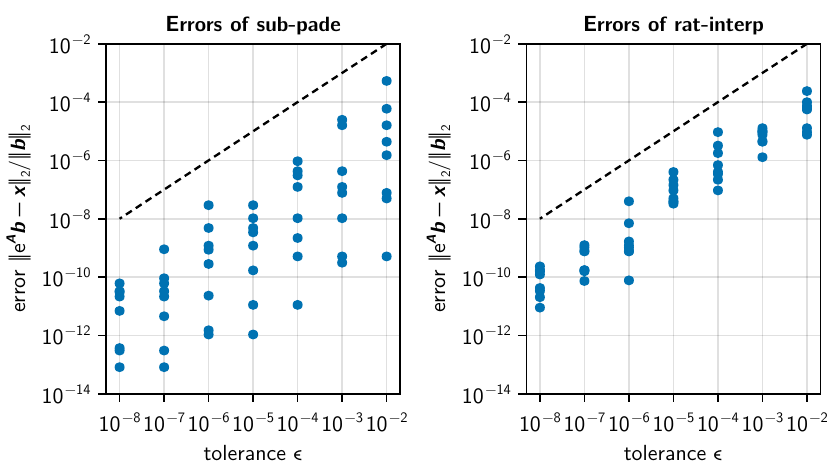}
  \caption{
    Errors of each method for $\tau = \bar{h}$.
    Each point represents the result for one test matrix.
    Although overestimations occur, the error is smaller than the tolerance for all matrices and methods.
  }
  \label{fig:result}
\end{figure}

To demonstrate the advantage of using $\calW(\hatbmA)$ over $\calW(\bmA)$, we compared the complexity of constructing rational functions $r(z)$ with $\epsilon=10^{-6}$ under two conditions:
(i) $|r(z)-\rme^{z}|\le \epsilon/(1+\sqrt{2})$ on the rectangle enclosing $\calW(\bmA)$, and
(ii) $|r(z)-\rme^{z}|\le \epsilon/((1+\sqrt{2})\kappa(\bmM)^{1/2})$ on the rectangle enclosing $\calW(\hatbmA)$.
Table \ref{tab:comparison_ns} lists the degrees of the denominators of rational approximations.
In all test cases, the degree required for (ii) is less than or equal to that for (i), illustrating the computational advantage gained by using $\calW(\hatbmA)$.

\begin{table}[htbp]
  \centering
  \caption{
    The degree of the denominator in the rational approximation when $\epsilon=10^{-6}$.
    Therefore, smaller numbers are preferable.
    For \texttt{sub-pade}, ``--'' indicates that the degree exceeded 320 ($s_\ast > 64$); for \texttt{rat-interp}, it indicates a degree exceeding 128.
    We note that for the \texttt{sub-pade}, the exponential function is approximated with the rational function with the form of $r(z) = [r_{4,5}(z/s)]^s$, and the degrees of the denominators of $r$ are listed in the table.
  }
  \begin{tabular}{ccc|rr|rr}
    \hline
    & & & \multicolumn{2}{c|}{\texttt{sub-pade}} & \multicolumn{2}{c}{\texttt{rat-interp}}\\
    Shape & Element & $d$ & (i) $\calW(\bmA)$ & (ii) $\calW(\hatbmA)$ & (i) $\calW(\bmA)$ & (ii) $\calW(\hatbmA)$ \\
    \hline
    square & P1 & $\num{e-1}$ & 25 & 25 & 9 & 9 \\
    square & P1 & $\num{e-3}$ & 10 & 10 & 5 & 5 \\
    square & P2 & $\num{e-1}$ & -- & 25 & -- & 11 \\
    square & P2 & $\num{e-3}$ & 25 & 20 & 8 & 8 \\
    star & P1 & $\num{e-1}$ & 20 & 20 & 11 & 9 \\
    star & P1 & $\num{e-3}$ & 15 & 15 & 6 & 6 \\
    star & P2 & $\num{e-1}$ & 185 & 25 & -- & 12 \\
    star & P2 & $\num{e-3}$ & 40 & 30 & 10 & 8 \\
    \hline
  \end{tabular}
  \label{tab:comparison_ns}
\end{table}

\subsection{Results for \texorpdfstring{$\rme^{\bmA}{\bmb}$ for $\tau=10\bar{h}$}{}}
We further evaluated the framework using a larger time step, $\tau = 10\bar{h}$, with the same matrices $\bmM$ and $\bmK$ as in the previous subsection.
The resulting errors are shown in Figure \ref{fig:result-10tau}.
With the exception of three cases using \texttt{rat-interp} at $\epsilon=10^{-8}$, the errors consistently fell below the tolerance.
In these three cases, the failure occurred during the construction of the partial fraction approximation itself, despite the use of double-double arithmetic.
This suggests the limitation lies in the rational approximation step rather than the proposed error control framework.

Table \ref{tab:comparison_ns_10tau} presents the degrees of the denominators for $\tau=10\bar{h}$.
The results indicate that while $\calW(\bmA)$ is often too large to construct an accurate rational approximation, $\calW(\hatbmA)$ provides a region to construct approximations successfully.

\begin{figure}[htbp]
  \centering
  \includegraphics[width=12cm]{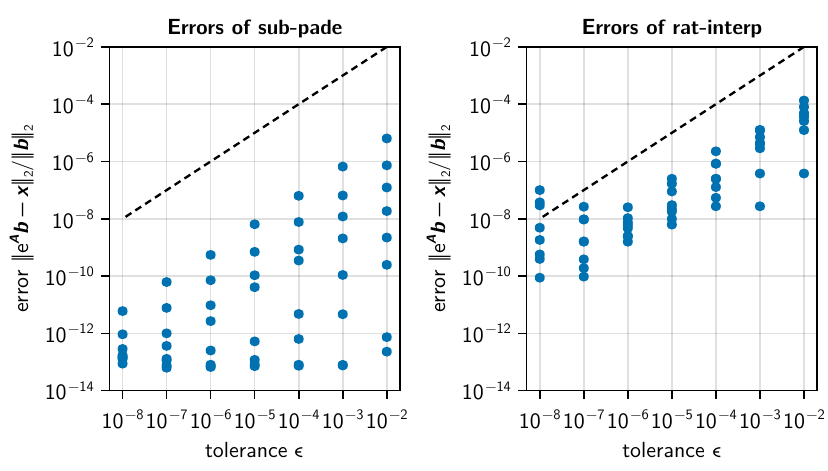}
  \caption{
    Approximation errors for $\tau = 10\bar{h}$.
    Each point represents the result for one test matrix.
  }
  \label{fig:result-10tau}
\end{figure}

\begin{table}[htbp]
  \centering
  \caption{
    Denominator degrees for $\epsilon=10^{-6}$ and $\tau=10\bar{h}$, following the format of Table \ref{tab:comparison_ns}.
  }
  \begin{tabular}{ccc|rr|rr}
    \hline
    & & & \multicolumn{2}{c|}{\texttt{sub-pade}} & \multicolumn{2}{c}{\texttt{rat-interp}}\\
    Shape & Element & $d$ & (i) $\calW(\bmA)$ & (ii) $\calW(\hatbmA)$ & (i) $\calW(\bmA)$ & (ii) $\calW(\hatbmA)$ \\
    \hline
    square & P1 & $\num{e-1}$ & 75 & 50 & 22 & 19  \\
    square & P1 & $\num{e-3}$ & 90 & 85 & 17 & 17  \\
    square & P2 & $\num{e-1}$ & -- & 130 & -- & 31  \\
    square & P2 & $\num{e-3}$ & -- & 170 & -- & 30  \\
    star & P1 & $\num{e-1}$ & 220 & 75 & 37 & 22  \\
    star & P1 & $\num{e-3}$ & 200 & 130 & 22 & 19  \\
    star & P2 & $\num{e-1}$ & -- & 190 & -- & 40  \\
    star & P2 & $\num{e-3}$ & -- & 310 & -- & 37  \\
    \hline
  \end{tabular}
  \label{tab:comparison_ns_10tau}
\end{table}

\subsection{Results for large matrices}\label{sect:test-largematrices}
In this section, some results using sparse implementations are presented.
The discretization was performed so that the matrix size is approximately $10,000$, and $\tau$ is set to $5\bar{h}$.
Information on the test matrices is listed in Table~\ref{tab:test_matrices_large}.
In the computations, the rectangular region $\calR$ and $\kappa(\bmM)$ were computed with \texttt{Arpack.jl} package that employs Lanczos method.
Here, the Lanczos method was stopped when the relative residual of the Ritz value is $10^{-3}$ or less.
The linear systems were solved using a sparse direct solver provided by \texttt{SparseArrays.jl} package.
As in previous experiments, the reference solution was computed using Julia's \texttt{exp} function.
Figure \ref{fig:result-large} shows the results.
As before, $\rme^{\bmA}\bmb$ was computed with an accuracy corresponding to the error tolerance in all cases.

\begin{table}[htbp]
  \centering
  \caption{
    Information on the test matrices $\bmA = \tau \bmM^{-1}\bmK$.
    Here, $n$ is the size of the matrix, and $\bar{h}$ denotes the average edge length of the mesh.
  }
  \begin{tabular}{ccc|rrrrr}
    \hline
    Shape & Element & $d$ & $n$ & $\bar{h}$ & $\kappa(\bmM)$ & $\|\bmA\|_2$ & $\|\rme^{\bmA}\bmb\|_2$ \\
    \hline
    square & P1 & \num{e-1} & 10000 & \num{1.13e-02} & \num{4.00e+00} & \num{7.31e+02} & \num{9.31e-01} \\
    square & P1 & \num{e-3} & 10000 & \num{1.13e-02} & \num{4.00e+00} & \num{3.72e+02} & \num{9.65e-01} \\
    square & P2 & \num{e-1} & 10201 & \num{2.23e-02} & \num{5.84e+00} & \num{1.90e+03} & \num{1.01e+00} \\
    square & P2 & \num{e-3} & 10201 & \num{2.23e-02} & \num{5.84e+00} & \num{9.46e+02} & \num{1.08e+00} \\
    star & P1 & \num{e-1} & 10504 & \num{2.55e-02} & \num{6.98e+00} & \num{7.67e+00} & \num{9.99e-01} \\
    star & P1 & \num{e-3} & 10504 & \num{2.55e-02} & \num{6.98e+00} & \num{9.54e+00} & \num{1.01e+00} \\
    star & P2 & \num{e-1} & 10627 & \num{5.08e-02} & \num{1.25e+01} & \num{2.00e+01} & \num{1.15e+00} \\
    star & P2 & \num{e-3} & 10627 & \num{5.08e-02} & \num{1.25e+01} & \num{2.63e+01} & \num{1.17e+00} \\
    \hline
  \end{tabular}
  \label{tab:test_matrices_large}
\end{table}

\begin{figure}[htbp]
  \centering
  \includegraphics[width=12cm]{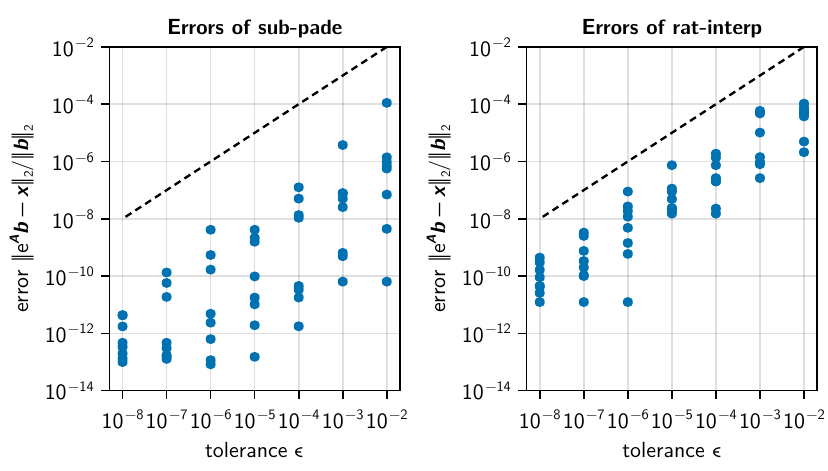}
  \caption{
    Approximation errors for matrices in Table \ref{tab:test_matrices_large} with $\tau=5\bar{h}$.
    Each point represents the result for one test matrix.
  }
  \label{fig:result-large}
\end{figure}

For reference, Table~\ref{tab:computational_time} lists the measured computational times for the third test matrix of Table~\ref{tab:test_matrices_large}.
Although it depends on the choice of eigensolvers and linear system solvers, the dominant part of the computation in the proposed framework would be the evaluation of $r(\bmA)\bmb$.
\begin{table}[htbp]
  \centering
  \caption{Computational time breakdown of \texttt{rat-interp} for the third test matrix  in Table~\ref{tab:test_matrices_large}.}
  \begin{tabular}{lr}
    \hline
    Operation & Time [s] \\
    \hline
    Extreme eigenvalues of $(\tau \bmD) \bmx = \mu \bmM \bmx$ & 0.12 \\
    Extreme eigenvalues of $(\tau \bmC) \bmx = \mu \bmM \bmx$ & 0.20 \\
    $\kappa(\bmM)$ & 0.07 \\
    Construction of $r$ & 0.54 \\
    $r(\bmA) \bmb$ & 109.07 \\
    \hline
  \end{tabular}
  \label{tab:computational_time}
\end{table}

\subsection{Comparison to rational Krylov methods}
Finally, for reference, we present a comparison with Krylov-subspace-based methods and \texttt{rat-interp}.
Here, computations were performed using the test matrix from the third row of Table~\ref{tab:test_matrices_large} in Subsection~\ref{sect:test-largematrices}.
We computed $\rme^{\bmA}\bmb$ using the rational Arnoldi method \cite[Alg.~3.1]{berljafa2015generalized} with the following choices of poles $\{\xi_k\}_{k=1}^m$:
\begin{itemize}
  \item $\xi_k = \infty$: Polynomial Arnoldi method.
  \item $\xi_k = 1, 10, 100$: Rational Arnoldi method repeatedly using a single pole.
  \item Poles chosen from a contour $\Gamma$ enclosing $\calW(\bmA)$ as proposed in \cite{gockler2017acceleration}.
  \item Poles chosen as the poles of the Pad\'{e} approximant $r_{4,5}$ as proposed in \cite{guttel2016scaled}.
\end{itemize}
The computational results are shown in Figure~\ref{fig:comparison}. A brief summary of the results and observations regarding the choice of poles is given below.
Although these cannot be compared directly with \texttt{rat-interp}, the relationship between the number of linear system solves and the error is also plotted in Figure~\ref{fig:comparison}.
\begin{itemize}
  \item The polynomial Krylov subspace method converged very slowly.
  \item The method using a single repeated pole converged rapidly when the pole was appropriately chosen.
  \item The method selecting poles from $\Gamma$ exhibited slow convergence in this experiment.
    This may be because the large advection term caused the contour to expand considerably, leading to poles with large absolute values and resulting in a convergence profile similar to that of the polynomial Krylov method.
  \item Using the poles of the Pad\'{e} approximant $r_{4,5}$ as the poles yielded fast convergence.
\end{itemize}
Overall, \texttt{rat-interp} can be either superior or inferior to rational Krylov subspace methods.
Since the most suitable solver depends on the specific scenario, performance evaluation on practical large problems remains an important subject for future work.

\begin{figure}[htbp]
  \centering
  \includegraphics[width=15cm]{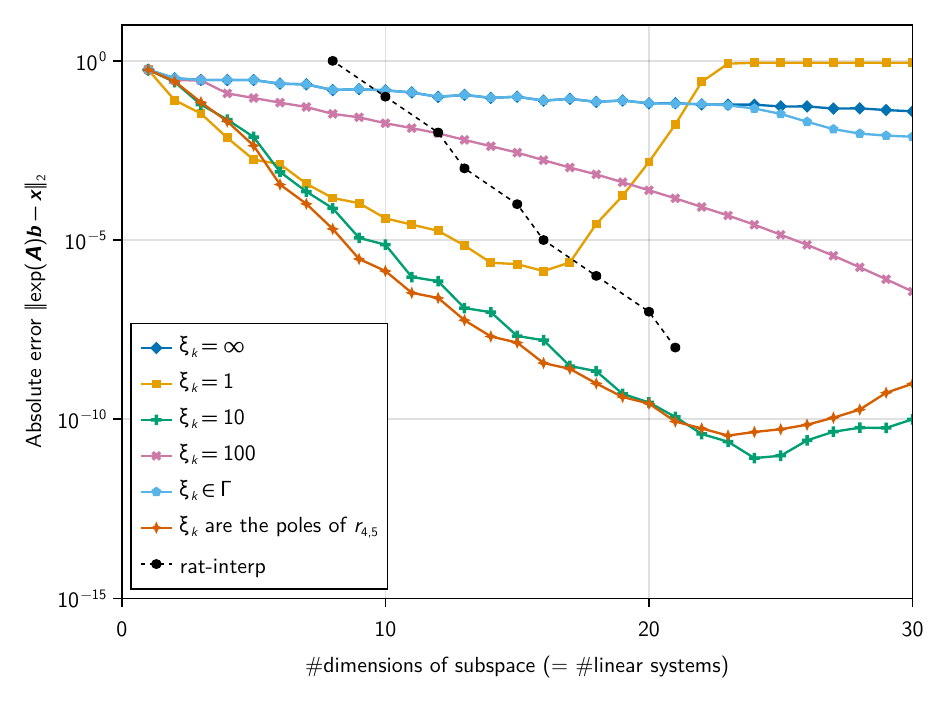}
  \caption{Convergence profile of the rational Krylov method for various poles $\xi_k$. For comparison, the results of \texttt{rat-interp} are plotted together with the number of linear system solves.}
  \label{fig:comparison}
\end{figure}

\section{Conclusion}
We proposed an error control framework for computing the matrix exponential of matrices with the structure $\bmA = \tau\bmM^{-1}\bmK$ that relies on $\calW(\hatbmA)$ rather than $\calW(\bmA)$.
The proposed set $\calW(\hatbmA)$ is convenient: an enclosing rectangle can be obtained numerically, and if $\calW(\bmK)$ lies in the left half plane then $\calW(\hatbmA)$ also resides in the left half plane.
Numerical experiments showed that, once a scalar rational approximation achieving the prescribed accuracy on $\calW(\hatbmA)$ is obtained, one can compute $\rme^{\bmA}\bmb$ with an error smaller than or equal to the tolerance.

Future work includes extending the proposed framework to the computation of matrix $\varphi$-functions and their linear combinations via augmented matrices commonly arising in exponential integrators, as well as comprehensive performance evaluations in parallel computing environments for large-scale practical problems.

\section*{Acknowledgements}
We are grateful to the anonymous referees for their insightful comments and constructive suggestions, which substantially improved our manuscript.
This work was partially supported by JSPS KAKENHI grant number 25H00449, 21K18301, 24K02951.


\begin{thebibliography}{10}

\bibitem{almohy2010new}
A.~H. {Al-Mohy} and N.~J. Higham.
\newblock A new scaling and squaring algorithm for the matrix exponential.
\newblock {\em SIAM J. Matrix Anal. Appl.}, 31(3):970--989, 2010.

\bibitem{almohy2011computing}
A.~H. {Al-Mohy} and N.~J. Higham.
\newblock Computing the action of the matrix exponential, with an application
  to exponential integrators.
\newblock {\em SIAM J. Sci. Comput.}, 33(2):488--511, 2011.

\bibitem{berljafa2015generalized}
M.~Berljafa and S.~G{\"u}ttel.
\newblock Generalized rational krylov decompositions with an application to
  rational approximation.
\newblock {\em SIAM J. Matrix Anal. Appl.}, 36(2):894--916, 2015.

\bibitem{blanes2026parallel}
S.~Blanes.
\newblock Parallel computation of functions of matrices and their action on
  vectors for exponential integrators.
\newblock {\em J. Comput. Appl. Math.}, 476:117090, 2026.

\bibitem{botchev2020art}
M.~A. Botchev and L.~A. Knizhnerman.
\newblock {ART}: Adaptive residual-time restarting for {Krylov} subspace matrix
  exponential evaluations.
\newblock {\em J. Comput. Appl. Math.}, 364:112311, 2020.

\bibitem{caliari2016leja}
M.~Caliari, P.~Kandolf, A.~Ostermann, and S.~Rainer.
\newblock The {{Leja}} method revisited: Backward error analysis for the matrix
  exponential.
\newblock {\em SIAM J. Sci. Comput.}, 38(3):A1639--A1661, 2016.

\bibitem{caliari2004interpolating}
M.~Caliari, M.~Vianello, and L.~Bergamaschi.
\newblock Interpolating discrete advection--diffusion propagators at {{Leja}}
  sequences.
\newblock {\em J. Comput. Appl. Math.}, 172(1):79--99, 2004.

\bibitem{crouzeix2017numerical}
M.~Crouzeix and C.~Palencia.
\newblock The numerical range is a $(1+\sqrt{2})$-spectral set.
\newblock {\em SIAM J. Matrix Anal. Appl.}, 38(2):649--655, 2017.

\bibitem{driscoll2025rational}
T.~A. Driscoll.
\newblock Rational function approximation in julia.
\newblock Zenodo, 2025.

\bibitem{driscoll2024aaa}
T.~A. Driscoll, Y.~Nakatsukasa, and L.~N. Trefethen.
\newblock {AAA} rational approximation on a continuum.
\newblock {\em SIAM J. Sci. Comput.}, 46(2):A929--A952, 2024.

\bibitem{eiermann2006restarted}
M.~Eiermann and O.~G. Ernst.
\newblock A restarted {Krylov} subspace method for the evaluation of matrix
  functions.
\newblock {\em SIAM J. Numer. Anal.}, 44(6):2481--2504, 2006.

\bibitem{gockler2013convergence}
T.~G{\"o}ckler and V.~Grimm.
\newblock Convergence analysis of an extended {{Krylov}} subspace method for
  the approximation of operator functions in exponential integrators.
\newblock {\em SIAM J. Numer. Anal.}, 51(4):2189--2213, 2013.

\bibitem{gockler2017acceleration}
T.~G{\"o}ckler and V.~Grimm.
\newblock Acceleration of contour integration techniques by rational {{Krylov}}
  subspace methods.
\newblock {\em J. Comput. Appl. Math.}, 316:133--142, 2017.

\bibitem{guttel2013rational}
S.~G{\"u}ttel.
\newblock Rational {{Krylov}} approximation of matrix functions: Numerical
  methods and optimal pole selection.
\newblock {\em GAMM-Mitt.}, 36(1):8--31, 2013.

\bibitem{guttel2016scaled}
S.~G{\"u}ttel and Y.~Nakatsukasa.
\newblock Scaled and squared subdiagonal {{Pad\'e}} approximation for the
  matrix exponential.
\newblock {\em SIAM J. Matrix Anal. Appl.}, 37(1):145--170, 2016.

\bibitem{hecht2012new}
F.~Hecht.
\newblock New development in {{FreeFEM}}++.
\newblock {\em J. Numer. Math.}, 20(3-4):251--266, 2012.

\bibitem{higham2005scaling}
N.~J. Higham.
\newblock The scaling and squaring method for the matrix exponential revisited.
\newblock {\em SIAM J. Matrix Anal. Appl.}, 26(4):1179--1193, 2005.

\bibitem{higham2008functions}
N.~J. Higham.
\newblock {\em Functions of Matrices: Theory and Computation}.
\newblock SIAM, Philadelphia, PA, 2008.

\bibitem{hochbruck2010exponential}
M.~Hochbruck and A.~Ostermann.
\newblock Exponential integrators.
\newblock {\em Acta Numer.}, 19:209--286, 2010.

\bibitem{ibanez2022two}
J.~Ib{\'a}{\~n}ez, J.~M. Alonso, P.~{Alonso-Jord{\'a}}, E.~Defez, and
  J.~Sastre.
\newblock Two {{Taylor}} algorithms for computing the action of the matrix
  exponential on a vector.
\newblock {\em Algorithms}, 15(2):48, 2022.

\bibitem{johnson1978numerical}
C.~R. Johnson.
\newblock Numerical determination of the field of values of a general complex
  matrix.
\newblock {\em SIAM J. Numer. Anal.}, 15(3):595--602, 1978.

\bibitem{saad1992analysis}
Y.~Saad.
\newblock Analysis of some {{Krylov}} subspace approximations to the matrix
  exponential operator.
\newblock {\em SIAM J. Numer. Anal.}, 29(1):209--228, 1992.

\bibitem{tatsuoka2024computing}
F.~Tatsuoka, T.~Sogabe, T.~Kemmochi, and S.-L. Zhang.
\newblock Computing the matrix exponential with the double exponential formula.
\newblock {\em Spec. Matrices}, 12(1):20240013, 2024.

\bibitem{weideman2007parabolic}
J.~Weideman and L.~Trefethen.
\newblock Parabolic and hyperbolic contours for computing the {{Bromwich}}
  integral.
\newblock {\em Math. Comput.}, 76(259):1341--1356, 2007.

\bibitem{weideman2006optimizing}
J.~A.~C. Weideman.
\newblock Optimizing Talbot's contours for the inversion of the {{Laplace}}
  transform.
\newblock {\em SIAM J. Numer. Anal.}, 44(6):2342--2362, 2006.

\end{thebibliography}

\appendix
\section{Details of obtaining a rational interpolation} \label{sect:rational_interpolation}
In this section, we describe the implementation details of the algorithm.
\paragraph{\texorpdfstring{Computational method for $r_{\mathrm{AAA}}$}{Computational method for the barycentric interpolation}}
\begin{itemize}
  \item The library RationalFunctionApproximation.jl \cite{driscoll2025rational} is an implementation of the AAA algorithm for continuum \cite{driscoll2024aaa} using the Julia language. It has a function that takes a polygonal region $\Omega$ on the complex plane, a function $f: \mathbb{C} \to \mathbb{C}$, and an error tolerance $\delta$ as input, and outputs a rational function $r_{\mathrm{AAA}}$ that approximates $f$ on $\Omega$, such that $\max_{z \in \Omega} |f(z) - r_{\mathrm{AAA}}(z)| \leq \delta$.
  Here, $r_{\mathrm{AAA}}$ is a rational function of type $[m,m]$ represented in barycentric form.
  That is, the library implicitly computes the weights $w_1, \dots, w_m$ and support points $\zeta_1, \dots, \zeta_m$, then it outputs $r_{\mathrm{AAA}}$.
  \item In the library, the degree $m$ is also dynamically determined such that the error is less than or equal to $\delta$.
  \item We obtained $r_{\mathrm{AAA}}$ using a program that imposes an additional constraint $m \geq 3$ in addition to the existing convergence criteria.
  \item Also, other parameters \texttt{refinement} and \texttt{stagnation} were set to 10.
\end{itemize}

\paragraph{Conversion from barycentric form to partial fraction form}
\begin{itemize}
  \item First, we computed the poles of $r_{\mathrm{AAA}}$. We used the \texttt{poles} function of RationalFunctionApproximation.jl.
  \item Using the \texttt{approximate} function of RationalFunctionApproximation.jl, we converted $r_{\mathrm{AAA}}$ into $r_{\mathrm{pf}}$.
  \begin{itemize}
    \item The \texttt{approximate} function takes a polygonal region $\Omega$, a function $f: \mathbb{C} \to \mathbb{C}$, and poles $\beta_1, \dots, \beta_m$ as input, and outputs $\gamma$ and $\alpha_k$ such that $r_{\mathrm{pf}}$ approximates $f$ on $\Omega$.
    \item We used the \texttt{DoubleFloats.jl} package for the double-double arithmetic.
  \end{itemize}
\end{itemize}

\section{Details of approximating the boundary of the numerical range} \label{sect:nr}
To illustrate the numerical range in Fig.~\ref{fig:numerical_range_test_matrices}, we used a Julia reimplementation of the \texttt{fv} function from The Matrix Computation Toolbox\footnote{\url{https://www.mathworks.com/matlabcentral/fileexchange/2360-the-matrix-computation-toolbox}}.
The theoretical background of the \texttt{fv} function is described in \cite{johnson1978numerical}.

The details of the procedure of the \texttt{fv} function are explained in Algorithm \ref{alg:numerical-range}.
\begin{algorithm}[htbp]
  \caption{Approximating the boundaries of numerical range}
  \label{alg:numerical-range}
  \begin{algorithmic}[1]
    \Statex \textbf{Input} $\bmA \in \bbC^{n \times n}$, and the number of mesh points $m$.
    \State $\partial \calW = \emptyset$
    \For{$j = 1, \dots, m$}
      \State $\theta_j = j \pi/ (m+1)$
      \State $\bmA_j = \rme^{\rmi \theta_j} \bmA$
      \State $\bmH_j = (\bmA_j + \bmA_j^\ast) / 2$
      \State Compute eigenvectors $\bmx_{\min}$ and $\bmx_{\max}$ corresponding to the minimum and the maximum eigenvalues of $\bmH_j$ respectively.
      \State $\partial \calW \gets \partial\calW \cup \{\bmx_{\min}^\ast \bmA \bmx_{\min}, \bmx_{\max}^\ast \bmA \bmx_{\max}\}$
    \EndFor
    \Statex \textbf{Output} $\partial \calW$, the approximation of the boundary of $\calW(\bmA)$
  \end{algorithmic}
\end{algorithm}

In our computation, $m$ is set to $128$, and eigenvectors are computed using Julia's dense matrix eigendecomposition function \texttt{eigen}.

\end{document}